\pdfoutput=1
\RequirePackage{ifpdf}
\ifpdf 
\documentclass[pdftex]{sigma}
\else
\documentclass{sigma}
\fi

\numberwithin{equation}{section}

\newtheorem{Theorem}{Theorem}[section]
\newtheorem{Corollary}[Theorem]{Corollary}
\newtheorem{Lemma}[Theorem]{Lemma}
\newtheorem{Proposition}[Theorem]{Proposition}
 { \theoremstyle{definition}
\newtheorem{Definition}[Theorem]{Definition}
\newtheorem{Example}[Theorem]{Example}
 }

\begin{document}

\allowdisplaybreaks

\newcommand{\arXivNumber}{1709.08394}

\renewcommand{\PaperNumber}{026}

\FirstPageHeading

\ShortArticleName{Contravariant Form on Tensor Product of Highest Weight Modules}

\ArticleName{Contravariant Form on Tensor Product\\ of Highest Weight Modules}

\Author{Andrey I.~MUDROV}

\AuthorNameForHeading{A.I.~Mudrov}

\Address{Department of Mathematics, University of Leicester, University Road, LE1 7RH Leicester, UK}
\Email{\href{mailto:am405@le.ac.uk}{am405@le.ac.uk}}

\ArticleDates{Received August 23, 2018, in final form March 25, 2019; Published online April 07, 2019}

\Abstract{We give a criterion for complete reducibility of tensor product $V\otimes Z$ of two irreducible highest weight modules $V$ and $Z$ over a classical or quantum semi-simple group in terms of a contravariant symmetric bilinear form on $V\otimes Z$. This form is the product of the canonical contravariant forms on $V$ and $Z$. Then $V\otimes Z$ is completely reducible if and only if the form is non-degenerate when restricted to the sum of all highest weight submodules in $V\otimes Z$ or equivalently to the span of singular vectors.}

\Keywords{highest weight modules; contravariant form; tensor product; complete reducibi\-lity}

\Classification{17B10; 17B37}

\rightline{\it In memoriam Ludwig Faddeev}

\section{Introduction}

Consider a semi-simple Lie algebra over the complex field and its classical or quantum universal enveloping algebra (assuming $q$ not a root of unity in the latter case). Fix a pair of irreducible highest weight modules. The question when their tensor product is completely reducible is a~fundamental problem of representation theory. That is always the case when the modules are finite-dimensional, due to the Weyl theorem, cf.~\cite{Erd} for classical and \cite{Jan} for quantum groups. However the problem is a challenge if one of them is infinite-dimensional. We give a criterion of complete reducibility in terms of a canonical contravariant form that is the product of the Shapovalov forms on the factors~\cite{Sh}. Quite amazingly, this result appears to be new.

Tensor products of representations of simple Lie algebras have been studied in detail in~\cite{BG, K}, where the main tool was the action of the center of their universal enveloping algebras. Our approach based on contravariant forms displays a relation with extremal projectors~\cite{AST,KT} and dynamical Weyl group~\cite{EV}. Unlike~\cite{BG}, we allow for arbitrary dimensions of the tensor factors. We simultaneously address the classical and quantum cases, although the exposition is provided for quantum groups. The classical case is a simplification that can be obtained by setting the deformation parameter $q\to 1$.

We prove that the tensor product of two irreducible modules $V\otimes Z$ of highest weight is completely reducible if and only if the canonical contravariant form is non-degenerate when restricted to the sum of all highest weight submodules in $V\otimes Z$. It is easy to see that the necessary condition here comes for free. Conversely, if the form is non-degenerate on the sum of highest weight submodules, they are all irreducible. It is then sufficient to prove that the sum exhausts all of $V\otimes Z$, which takes a certain effort.

Equivalently, $V\otimes Z$ is completely reducible if and only if the form is non-degenerate on the subspace of singular vectors in $V\otimes Z$, i.e., annihilated by positive Chevalley generators. Having parameterized them with elements from either of reciprocal ``extremal subspaces" $V^+_Z\subset V$ and $Z^+_V\subset Z$, we construct operators $\theta_{V,Z}$, $\theta_{Z,V}$ sending $V^+_Z$ and $Z^+_V$ to their duals under the Shapovalov forms on $V$ and $Z$, respectively. These parametrisations link the canonical form with Zhelobenko cocycles~\cite{Zh} and lay a ground for developing computational methods.

One application of the presented result is quantization of equivariant vector bundles over semi-simple conjugacy classes and coadjoint orbits regarded as projective modules over quantized function algebras.
The infinite-dimensional modules of interest are parabolic or generalized parabolic modules generating module categories over the category of finite-dimensional representations. A quantim conjugacy class or a coadjoint orbit is realized by linear endomorphisms on such a module whose tensor product with finite-dimensional modules supports ``representations'' of a trivial bundle. Its non-trivial sub-bundles correspond to invariant idempotents that project the tensor product onto submodules~\cite{M1,M2}.

\section{Preliminaries}\label{SecPrelim}
Suppose that $\mathfrak{g}$ is a complex semi-simple Lie algebra and $\mathfrak{h}\subset \mathfrak{g}$ its Cartan subalgebra. Fix a~triangular decomposition $\mathfrak{g}=\mathfrak{g}_-\oplus \mathfrak{h}\oplus \mathfrak{g}_+$ with nilpotent Lie subalgebras $\mathfrak{g}_\pm$. Denote by $\mathrm{R}$ the root system of $\mathfrak{g}$, and by $\mathrm{R}^+$ the subset of positive roots with basis $\Pi^+$ of simple roots. Choose an inner product $(\cdot,\cdot)$ on $\mathfrak{h}$ as a multiple of the restricted Killing form and transfer it to the dual space $\mathfrak{h}^*$. For all $\lambda\in \mathfrak{h}^*$ denote by $h_\lambda$ the element of $\mathfrak{h}$ such that $\mu(h_\lambda)=(\mu,\lambda)$ for all $\mu\in \mathfrak{h}^*$.

By $U_q(\mathfrak{g})$ we understand the standard quantum group \cite{ChP, D3}, that is a complex unital algebra with the set of generators $e_\alpha$, $f_\alpha$, and $q^{\pm h_\alpha}$ satisfying
\begin{gather*}
q^{\pm h_\alpha}e_\beta=q^{\pm (\alpha,\beta)}e_\beta q^{\pm h_\alpha}, \qquad [e_\alpha,f_\beta]=\delta_{\alpha,\beta}[h_\alpha]_q,\\
q^{\pm h_\alpha}f_\beta=q^{\mp (\alpha,\beta)}f_\beta q^{\pm h_\alpha},\qquad \alpha \in \Pi^+,
\end{gather*}
plus the Serre relations among $\{e_\alpha\}_{\alpha\in \Pi^+}$ and $\{f_\alpha\}_{\alpha\in \Pi^+}$, see \cite{ChP} for details. Here $[h_\alpha]_q=\frac{q^{h_\alpha}-q^{-h_\alpha}}{q-q^{-1}}$, and $q^{h_\alpha}q^{-h_\alpha}=1=q^{-h_\alpha}q^{h_\alpha}$. The complex parameter $q\not =0$ is assumed not a~root of unity.

Fix the comultiplication in $U_q(\mathfrak{g})$ as
\begin{gather*}
\Delta(f_\alpha)= f_\alpha\otimes 1+q^{-h_\alpha}\otimes f_\alpha,\qquad\!\! \Delta\big(q^{\pm h_\alpha}\big)=q^{\pm h_\alpha}\otimes q^{\pm h_\alpha},\qquad\!\!\Delta(e_\alpha)= e_\alpha\otimes q^{h_\alpha}+1\otimes e_\alpha.
\end{gather*}
Then the antipode acts on the generators by the assignment $\gamma( f_\alpha)=- q^{h_\alpha}f_\alpha$, $\gamma\big( q^{\pm h_\alpha}\big)=q^{\mp h_\alpha}$, $\gamma( e_\alpha)=- e_\alpha q^{-h_\alpha}$. The counit homomorphism returns $\epsilon(e_\alpha)=\epsilon(f_\alpha)=0$, and $\epsilon\big(q^{\pm h_\alpha}\big)=1$.

Denote by $U_q(\mathfrak{h})$, $U_q(\mathfrak{g}_+)$, $U_q(\mathfrak{g}_-)$ the subalgebras in $U_q(\mathfrak{g})$ generated by $\big\{q^{\pm h_\alpha}\big\}_{\alpha\in \Pi^+}$,\linebreak $\{e_\alpha\}_{\alpha\in \Pi^+}$, and $\{f_\alpha\}_{\alpha\in \Pi^+}$, respectively. The quantum Borel subgroups $U_q(\mathfrak{b}_\pm)=U_q(\mathfrak{g}_\pm)U_q(\mathfrak{h})$ are Hopf subalgebras in $U_q\!(\mathfrak{g})$. We introduce a grading on $U_q\!(\mathfrak{g}_\pm\!)$ by setting \smash{$\deg(e_\alpha)\!=\!1\!=\!\deg(f_\alpha)$}.

We will need the following involutive maps on $U_q(\mathfrak{g})$.
The assignment
\begin{gather*}
\sigma\colon \ e_\alpha\mapsto f_\alpha, \qquad\sigma\colon \ f_\alpha\mapsto e_\alpha, \qquad \sigma\colon \ q^{h_\alpha}\mapsto q^{-h_\alpha}
\end{gather*}
gives rise to an algebra automorphism of $U_q(\mathfrak{g})$. It is also a coalgebra anti-automorphism. The assignment
\begin{gather*}
\tilde \omega\colon \ e_\alpha\mapsto f_\alpha, \qquad\tilde \omega\colon \ f_\alpha\mapsto e_\alpha, \qquad \tilde \omega\colon \ q^{h_\alpha}\mapsto q^{h_\alpha}
\end{gather*}
for all $\alpha \in \Pi^+$ extends to an algebra anti-automorphism of $U_q(\mathfrak{g})$. This map flips the subalgebras~$U_q(\mathfrak{g}_+)$ and $U_q(\mathfrak{g}_-)$ but it is not compatible with the comultiplication.

We will also work with the involutive map $\omega =\gamma^{-1}\circ \sigma$, which is an algebra anti-automorphism and preserves comultiplication. However, it does not exchange $U_q(\mathfrak{g}_+)$ and $U_q(\mathfrak{g}_-)$ as $\tilde \omega$ does. To fix this disadvantage, we have to pass to the Borel subalgebras. Let us show that replacement of~$\tilde \omega$ with $\omega$ does not affects certain left/right ideals in $U_q(\mathfrak{b}_\pm)$ which are of interest for the present exposition.

An element $\psi\in U_q(\mathfrak{g})$ is said to be of weight $\lambda\in \mathfrak{h}^*$ if $q^{h_\alpha}\psi q^{-h_\alpha}=q^{(\lambda,\alpha)}\psi$. Such weights are linear combinations of simple roots with integer coefficients.

\begin{Lemma}\label{eq-ideals} The map $f_{\alpha}\mapsto c f_{\alpha}q^{\pm h_\alpha}$, with non-zero $c\in \mathbb{C}$, extends to an automorphism of the Borel subalgebra $U_q(\mathfrak{b}_-)$. Under this map, any element $\psi\in U_q(\mathfrak{b}_-)$ of a given weight is sent to $\psi \phi$ for some $\phi\in U_q(\mathfrak{h})$.
\end{Lemma}

\begin{proof} Since the Serre relations are homogeneous, cf.~\cite{ChP}, it is sufficient to assume that $\psi$ is a~monomial in the Chevalley generators and check that $\phi$ depends only on the weight of~$\psi$. Let $\beta_1,\ldots, \beta_m$ be a~sequence of simple positive roots and take $\psi=f_{\beta_1}\cdots f_{\beta_{m}}$. Then $\phi=c^m q^{\mp \sum_{i<j}(\beta_i,\beta_j)}q^{\sum_{i=1}^{m}h_{\beta_i}}$.
\end{proof}
\begin{Corollary}\label{inv-ideals}Let $I\in U_q(\mathfrak{g}_+)$ be a $U_q(\mathfrak{h})$-graded left ideal. Then the right ideals in $U_q(\mathfrak{b}_-)$ generated by $\tilde \omega(I)$ and by $\omega(I)$ coincide.
\end{Corollary}
\begin{proof} A $U_q(\mathfrak{h})$-graded left ideal is generated by elements of certain weights. Now the proof follows from Lemma~\ref{eq-ideals}.
\end{proof}

\section{Extremal vectors in tensor product}
In this preparatory section we fix some notation and describe two parametrizations of singular vectors in $V\otimes Z$ by hom-sets relative to the subalgebra $U_q(\mathfrak{g}_+)$. We will use those parameterizations for constructing pullbacks of a contravariant form on $V\otimes Z$.

For every $U_q(\mathfrak{h})$-module $V$ we denote by $\Lambda(V)$ the set of its weights, i.e., $\mu\in \Lambda(V)$ if there is a non-zero $v\in V$ such that $q^{h_{\alpha}}v=q^{(\alpha,\mu)}v$, $\forall\, \alpha$. The set of $v\in V$ that satisfy this condition (including the zero vector) is denoted by $V[\mu]$.

Let $\Gamma=\mathbb{Z}\Pi^+$ denote the root lattice and set $\Gamma_+=\mathbb{Z}_+\Pi^+$ ($\mathbb{Z}_+$ stands for the monoid of non-negative integers). For $\alpha,\beta \in \Lambda(V)$ we write $\alpha\prec \beta$ or $\beta \succ \alpha$ if $\beta-\alpha \in \Gamma_+\backslash \{0\}$.

$V$ is called a highest weight module if there is a vector $1_\nu$ of weight $\nu\in \Lambda(V)$ such that $V$ is cyclically generated by $v$ and $e_\alpha 1_\nu=0$ for all $\alpha \in \Pi$. Then $1_\nu$ is called highest vector and $\nu$ the highest weight. $V$ is also cyclicly generated by $1_\nu$ as a module over $U_q(\mathfrak{g}_-)$. It is a quotient of the Verma module of the same highest weight $\nu$, which is freely generated over $U_q(\mathfrak{g}_-)$ by its highest vector. Clearly $\mu\prec \nu$ for all $\mu\in \Lambda(V)$ distinct from $\nu$.

Recall that a non-zero weight vector of a $U_q(\mathfrak{g})$-module $V$ is called singular (extremal) if it is killed by all positive generators of~$U_q(\mathfrak{g})$. A singular vector $u$ of weight $\lambda$ defines a homomorphism of the corresponding Verma module extending the assignment $1_\lambda\mapsto u$ on its highest vector. The submodule in $V$ generated by a singular vector $u$ is a highest weight module with the highest vector $u$. We call it a highest weight submodule.

By $V^*$ we mean the $U_q(\mathfrak{h})$-graded restricted (left) dual to a highest weight $U_q(\mathfrak{g})$-module $V$. It is isomorphic to the right dual ${}^*\!V$ since $U_q(\mathfrak{g})$ is quasi-triangular. If $V$ is irreducible, then $V^*$ is cyclically generated by a vector of lowest weight, i.e., the one annihilated by all~$f_\alpha$, $\alpha \in \Pi$.

In what follows, $V$ and $Z$ are highest weight modules. They are assumed irreducible unless otherwise is explicitly stated. We reserve notation $\nu$ for the highest weight of $V$ and $\zeta$ for the highest weight of $Z$.
By $1_\nu$ and $1_\zeta$ we denote their highest weight generators. The vector subspace in $V\otimes Z$ spanned by singular vectors is denoted by $(V\otimes Z)^+$.

Denote by $V\circledast Z$ the sum of all highest weight submodules in $V\otimes Z$; it is generated by $(V\otimes Z)^+$. Let $u=\sum_{i}v_i\otimes z_i\in (V\otimes Z)^+$ be a singular vector, where $v_i$ and $z_i$ are weight vectors. Assuming the set of $v_i$ independent, we call $\{z_i\}$ $Z$-coefficients of $u$, with the lowest weight coefficient corresponding to~$1_\nu$. We define the $V$-coefficients of $u$ similarly.

\begin{Lemma}\label{duality}Singular vectors in $V\otimes Z$ are in bijection with $\mathrm{Hom}_{U_q(\mathfrak{g}_+)}(V^*,Z)$ and, alternatively, with $\mathrm{Hom}_{U_q(\mathfrak{g}_+)}({}^*\!Z,V)$.
\end{Lemma}
\begin{proof}Suppose that $\{v_i\}$ form a basis in $V$. Let $\{v^i\}\subset V^*$ be the dual basis and $\{z_i\}\subset Z$ the set of coefficients for a~given singular vector $u$. Then the map $V^*\to Z$, $v^i\mapsto z_i$, is a~$U_q(\mathfrak{g}_+)$-module homomorphisms. It is clear that every singular vector is obtained this way. The case of $\mathrm{Hom}_{U_q(\mathfrak{g}_+)}({}^*\!Z,V)$ is similar.
\end{proof}

Obviously a $Z$-coefficient with a maximal weight in the expansion of $u$ is killed by $\{e_\alpha\}_{\alpha \in \Pi^+}$ and is singular in $Z$. Therefore, for irreducible $Z$, such a coefficient is equal to $1_\zeta$, up to a~non-zero scalar factor. Reciprocally, the $V$-coefficient of a maximal weight is singular in $V$ and therefore is a non-zero scalar multiple of $1_\nu$.

Let $I^-_Z\subset U_q(\mathfrak{g}_-)$ be the annihilator of $1_\zeta\in Z$ and put $I^+_Z=\sigma(I^-_Z)\subset U_q(\mathfrak{g}^+)$. We define left ideals $I^\pm_V \subset U_q(\mathfrak{g}^\pm)$ similarly. Denote by $V^+_Z\subset V$ and $Z^+_V\subset Z$ the kernels of $I^+_Z$ and $I^+_V$, respectively.
\begin{Proposition}\label{V-Z-extr}The vector subspace $V^+_Z$ is isomorphic to $Z^+_V$ via $V^+_Z\simeq (V\otimes Z)^+\simeq Z^+_V$.
\end{Proposition}
\begin{proof}Since $V$ and $Z$ are irreducible, their dual modules are cyclicly generated by the lowest vectors. Therefore $\mathrm{Hom}_{U_q(\mathfrak{g}_+)}({}^*\!Z,V)=V^+_Z$ and $\mathrm{Hom}_{U_q(\mathfrak{g}_+)}({}^*\!V,Z)=Z^+_V$, so one is left to apply Lemma~\ref{duality}.
\end{proof}

Observe that the isomorphisms are implemented by the assignments $u\mapsto u_1 (u_2, {}^*\!1_{\zeta})\in V^+_Z$, and $u\mapsto ( u_1, {}^*\!1_{\nu}) u_2 \in Z^+_V$, where $u=u^1\otimes u^2\in (V\otimes Z)^+$ in the Sweedler notation. These assignments involve the invariant pairing with the lowest vectors ${}^*\!1_{\nu}$, ${}^*\!1_{\zeta}$ (of weights~$-\nu$ and~$-\zeta$) of the dual modules and return the coefficients of lowest weights. We call them leading coefficients. Other coefficients in the expansion of a given singular vectors are obtained from them by the action of~$U_q(\mathfrak{g}_+)$.

We denote the inverse isomorphisms by $\delta_l \colon V^+_Z\to (V\otimes Z)^+$ and $\delta_r\colon Z^+_V\to (V\otimes Z)^+$.

\begin{Example}Suppose $\mathfrak{k}$ is a Levi subalgebra in $\mathfrak{g}$ and let $\Pi^+_\mathfrak{k}\subset \Pi^+_\mathfrak{g}$ be its basis of simple roots. Consider $Z$ to be a scalar parabolic Verma module corresponding to a one-dimensional representation of $U_q(\mathfrak{k})$. The left ideal $I^+_Z$ is then generated by $\{e_\alpha\}_{\alpha\in \Pi^+_\mathfrak{k}}$, and $V^+_Z$ becomes the subspace of $U_q(\mathfrak{k})$-singular vectors in $V$.
\end{Example}

\section{Contravariant forms}
Let $V$ be a $U_q(\mathfrak{g})$-module. A symmetric bilinear form $\langle \cdot,\cdot\rangle$ on $V$ is called contravariant if for all $v,w\in V$ and all $x\in U_q(\mathfrak{g})$ it satisfies $\langle x v, w\rangle=\langle v,\omega(x)w\rangle$. The following properties of contravariant forms readily follow.
\begin{Proposition}\label{form_features} \quad
\begin{enumerate}\itemsep=0pt
 \item[$i)$] The kernel of the form is a submodule in $V$.
 \item[$ii)$] Weight subspaces $V[\mu]$ and $V[\nu]$ for different $\mu,\nu\in \Lambda(V)$ the are orthogonal.
 \item[$iii)$] Two highest weight submodules are orthogonal if and only if their canonical generators are orthogonal.
 \end{enumerate}
\end{Proposition}

\looseness=-1 It is known that every module of highest weight is equipped with a unique, up to a scalar multiplier, contravariant form, which is non-degenerate if and only if the module is irreducible. When~$V$ is a Verma module, it is called Shapovalov form~\cite{Sh}, which name we apply to all highest weight modules. The Shapovalov form is equivalent to an invariant pairing $V\otimes {}^*\!V\to \mathbb{C}$ via the linear isomorphism $\sigma \colon V\to {}^*\!V$. We will work with both contravariant forms and invariant pairings.

Now we return to the settings of the previous section.
\begin{Definition} Let $V$ and $Z$ be two highest weight modules. By {\em canonical contravariant form} on $V\otimes Z$ we mean the product of the Shapovalov forms on $V$ and $Z$.
\end{Definition}

\looseness=-1 Denote by $V^\perp_Z\subset V$ the vector subspace $\omega(I^+)V$, i.e., the image of the right ideal $\omega(I^+_Z)\subset U_q(\mathfrak{b}_-)$. By Lemma~\ref{inv-ideals}, it is equal to image of the right ideal $\tilde \omega (I^+_Z)\subset U_q(\mathfrak{g}_-)$. It is exactly the annihilator of $V^+_Z$, so the quotient ${}^+\!V_Z=V/V^\perp_Z$ is isomorphic to the dual vector space to $V^+_Z$ with respect to the form. We identify ${}^+\!V_Z$ with a fixed vector subspace in $V$ that is transversal to $V^\perp_Z$.

Define extremal twist operator $\theta_{V,Z}\colon V^+_Z\to {}^+\!V_Z$ as follows. For $v\in V^+_Z[\xi]$ write $u=\delta_l(v)$ as $u=\sum_{\alpha\in \Gamma^+}\sum_i v_{\alpha,i}\otimes f_{\alpha,i} 1_\zeta$, where $v_{\alpha,i}\in V$ are vectors of weight $\xi+\alpha$ and $f_{\alpha,i}\in U_q(\mathfrak{g}_-)$ are some elements of weight $-\alpha$. Then put $\theta_{V,Z}(v)$ equal to the image of $\sum_{\alpha\in \Gamma^+}\sum_{i}\gamma^{-1}(f_{\alpha,i})v_{\alpha,i}$ in ${}^+\!V_Z$ projected along $V^\perp_Z$. This image is independent of the choice of $f_{\alpha,i}$. Indeed, the if $f'_{\alpha,i}\in U_q(\mathfrak{g}_-)$ is another element such that $f_{\alpha,i}'1_\zeta=f_{\alpha,i}1_\zeta$, then $\psi_{\alpha,i}=f_{\alpha,i}'-f_{\alpha,i}$ belongs to $I^-_Z$, and $\gamma^{-1}(\psi_{\alpha,i})=(\omega \circ \sigma) (\psi_{\alpha,i})\in \omega(I^+_Z)$. Therefore $f_{\alpha,i}'v_i=f_{\alpha,i}v_i \mod V^\perp_Z$ as required. Further on we will suppress the index $i$ replacing $\sum_iv_{\alpha,i}\otimes f_{\alpha,i}$ with $ v_{\alpha}\otimes f_{\alpha}$ in the Sweedler-like fashion.

Recall that there are natural $U_q(\mathfrak{g}_\pm)$-module epimorphisms $U_q(\mathfrak{g}_+)\to {}^*\!Z$ and $U_q(\mathfrak{g}_-)\to Z$, where the algebras act on themselves by left multiplication. Extend the tensor product ${}^*\!Z\otimes Z$ by infinite formal sums $\sum_{\xi\in \Lambda({}^*\!Z), \zeta\in \Lambda(Z)}{}^*\!Z[\xi]\hat \otimes Z[\zeta]$. Let $\mathcal{S}$ denote the $U_q(\mathfrak{g})$-invariant image of $1$ under the coevaluation morphism $\mathbb{C}\to {}^*\!Z\hat \otimes Z$. Using a linear section of the module map $U_q(\mathfrak{g}_+)\hat \otimes U_q(\mathfrak{g}_-)\to {}^*\!Z\hat \otimes Z$, we pick up a lift $\mathcal{F}\in U_q(\mathfrak{g}_+)\hat \otimes U_q(\mathfrak{g}_-)$ of~$\mathcal{S}$.
\begin{Lemma}The map $\delta_l\colon V^+_Z\to (V\otimes Z)^+$ acts by the assignment $v\mapsto \mathcal{F}_1v\otimes \mathcal{F}_21_\zeta$, in the Sweedler symbolic notation for tensor factors.
\end{Lemma}
\begin{proof} Indeed, ${}^*\!Z\hat \otimes Z$ is a $U_q(\mathfrak{g})$-module, and the inverse form $\mathcal{S}\in {}^*\!Z\hat \otimes Z$ is $U_q(\mathfrak{g})$-invariant. The homomorphism ${}^*\!Z \otimes Z\to V\otimes Z$ of $U_q(\mathfrak{g}_+)$-modules determined by $v\in V^+_Z$ extends to a~homomorphism ${}^*\!Z\hat \otimes Z\to V\otimes Z$ that sends $\mathcal{S}$ to $u=\mathcal{F}_1v\otimes \mathcal{F}_21_\zeta$. Therefore $u$ is $U_q(\mathfrak{g}_+)$-invariant and hence singular. Evaluating the invariant inner product of the $Z$-component of $\mathcal{F}_1v\otimes \mathcal{F}_21_\zeta$ with ${}^*\!1_\zeta$ we get $v$, which proves the statement.
\end{proof}

The element $\Upsilon_Z = \gamma^{-1}(\mathcal{F}_2)\mathcal{F}_1$ belongs to a certain extension of $U_q(\mathfrak{g})$ and defines a $U_q(\mathfrak{h})$-linear map $V^+_Z\to V$. Remark that $\Upsilon_Z$ controls the transformation of the antipode under the quasi-Hopf twist of $U_q(\mathfrak{g})$ by $\mathcal{F}$~\cite{D1}. These elements for $Z$ a Verma module were introduced in~\cite{EV} where they participated in construction of dynamical Weyl group. For the case of parabolic modules they appeared in \cite{KM} in connection with dynamical twist over a non-abelian base.

Observe that right hand side of the equality $\delta_l(v)=\mathcal{F}(v\otimes 1_\zeta)$, $v\in V^+_Z$, is independent of the choice of $\mathcal{F}$. Any other lift differs from $\mathcal{F}$ by an element from $U_q(\mathfrak{g})\otimes I_Z^-+ I^+_Z \otimes U_q(\mathfrak{g})$, which kills $V^+_Z\otimes 1_\zeta$. For all $v,w\in V^+_Z$, one has $\langle \theta_{V,Z}(v),w\rangle=\langle \Upsilon_Z(v),w\rangle$. The right-hand side here is independent of a particular lift, because $U_q(\mathfrak{g})I^+_Z$ kills $V^+_Z$ and $\gamma^{-1}(I^-_Z)U_q(\mathfrak{g})V= \omega(I^+_Z)V$ is orthogonal to $V^+_Z$.

\begin{Proposition} The form $\langle \theta_{V,Z}(\cdot),\cdot \rangle$ is the pullback of the canonical form restricted to \linebreak $(V \otimes Z)^+$ under the isomorphism $V^+_Z\to (V\otimes Z)^+$.
\end{Proposition}
\begin{proof}Choose extremal vectors $u, \tilde u \in (V\otimes Z)^+$ of weight $\eta +\zeta$ and $\xi+\zeta$, respectively. Expanding them over the positive hull of the root lattice we write $u=\sum_{\alpha\geqslant0}v_{\eta+\alpha}\otimes f_\alpha1_\zeta =\delta_l(v_\eta)$ and $\tilde u=\sum_{\beta\geqslant0} \tilde v_{\xi+\beta}\otimes \tilde z_{\beta-\zeta}=\delta_l(\tilde v_\xi)$ using the Sweedler-like convention for suppressing summation. Then{\samepage
\begin{gather*}
\langle u, \tilde u\rangle =\sum_{\alpha,\beta\geqslant0}\langle v_{\eta+\alpha},\tilde v_{\xi+\beta}\rangle \langle f_\alpha 1_\zeta, \tilde z_{\beta-\zeta}\rangle
=\sum_{\alpha,\beta\geqslant0}\langle v_{\eta+\alpha},\tilde v_{\xi+\beta}\rangle \big\langle 1_\zeta, \gamma^{-1}\circ \sigma(f_\alpha)\tilde z_{\beta-\zeta}\big\rangle\\
\hphantom{\langle u, \tilde u\rangle}{} =\sum_{\alpha,\beta\geqslant0}\langle v_{\eta+\alpha},\sigma(f_\alpha) \tilde v_{\xi+\beta}\rangle \langle 1_\zeta, \tilde z_{\beta-\zeta}\rangle
=\left\langle \sum_{\alpha\geqslant0}\gamma^{-1}(f_\alpha) v_{\eta+\alpha},\tilde v_{\xi}\right\rangle
=\langle \theta_{V,Z}(v_{\eta}),\tilde v_{\xi}\rangle
\end{gather*}
as required.}
\end{proof}

\section[Height filtration on $V\otimes Z$]{Height filtration on $\boldsymbol{V\otimes Z}$}

Since all weights in a module $V$ of highest weight $\nu$ belong to $\nu-\Gamma^+$ we can speak of height $\mathrm{ht}(\xi)\in \mathbb{Z}_+$ of $\xi\in \Lambda(V)$ defining it as the sum of coordinates of $\nu-\xi$ with respect to the basis~$\Pi^+$. We define $\mathrm{ht}(\mu)$ for a weight $\mu\in \Lambda(V\otimes Z)$ similarly by setting it equal to the sum of coordinates of $\nu+\zeta-\mu\in \Gamma_+$. A vector is said to be of height $k$ if it carries a weight of the corresponding height. The assignment $V\ni v\mapsto v\otimes 1_\zeta\in V\otimes Z$ is height-preserving.

For $k\in \mathbb{Z}_+$ denote by $(V\otimes Z)_k$ the submodule in $V\otimes Z$ generated by vectors of height $\leqslant k$ from $V\otimes 1_\zeta$. The submodules $(V\otimes Z)_k$ form an increasing filtration of $V\otimes Z$.
\begin{Lemma}\label{diagonal}Fix $\xi\in \Lambda(V)$, $\eta\in \Lambda(Z)$, $\beta\in -\Gamma_+$ and put $\mu=\xi+\beta+\eta \in\Lambda(V\otimes Z)$. Then for every $\psi \in U_q(\mathfrak{g}_-)[\beta]$, $v_\xi\in V[\xi]$, and $z_\eta\in Z[\eta]$, one has $\gamma^{-1}(\psi) v_\xi\otimes z_\eta = v_\xi\otimes \psi z_\eta \mod (V\otimes Z)_{k-1}$ with $k=\mathrm{ht}(\mu)$.
\end{Lemma}
\begin{proof}Remark that, in particular, this implies $v_\xi\otimes \psi 1_\zeta \in (V\otimes Z)_{k}$. So we can consider a~monomial $\psi$ in negative Chevalley generators and do induction on the degree of $\psi$.

For $\psi=1$ the statement is trivial. Suppose it is true for all monomials of degree $m-1>0$ and present $\psi$ of degree $m$ as $\psi= f_{\alpha} \psi'$ for some $\alpha \in \Pi^+$. Here $\psi'\in U_q(\mathfrak{g}_-)$ has
degree $m-1$. Then $v_\xi\otimes \psi' 1_\zeta \in (V\otimes Z)_{k-1}$ by the induction assumption and the above remark. The action of $\Delta(f_\alpha)=- \gamma^{-1}(f_{\alpha} )q^{-h_\alpha}\otimes 1+q^{-h_\alpha}\otimes f_\alpha$ on the tensor $v_\xi\otimes \psi' 1_\zeta\in (V\otimes Z)_{k-1}$ yields
\begin{gather*}
 v_\xi\otimes\psi 1_\zeta= v_\xi\otimes f_\alpha \psi' 1_\zeta=\gamma^{-1}(f_{\alpha}) v_\xi\otimes \psi' 1_\zeta \qquad \mod (V\otimes Z)_{k-1}.
\end{gather*}
The rightmost term transforms to $\gamma^{-1}(\psi')\gamma^{-1}(f_{\alpha}) v_\xi\otimes 1_\zeta=\gamma^{-1}(\psi) v_\xi\otimes 1_\zeta \mod(V\otimes Z)_{k-1}$, by the induction assumption. This completes the proof.
\end{proof}

Note that this lemma is true for arbitrary highest weight modules $V$ and $Z$, not necessarily irreducible.
\begin{Corollary}\label{two-filtrations}\quad
\begin{enumerate}\itemsep=0pt
 \item[$i)$] All vectors in $V\otimes Z$ of height $\leqslant k$ belong to $(V\otimes Z)_{k}$.
\item[$ii)$] Vectors from $V^\perp_Z\otimes 1_\zeta$ of height $k$ belong to $(V\otimes Z)_{k-1}$.
 \item[$iii)$] Vectors from ${}^+\!V_Z\otimes 1_\zeta$ of height $k$ generate $(V\otimes Z)_k$ modulo $(V\otimes Z)_{k-1}$.
\item[$iv)$] For all $v\in V^+_Z$ of height $k$, $\theta_{V,Z}(v)\otimes 1_\zeta=\delta_l(v)$ modulo $(V\otimes Z)_{k-1}$.
\end{enumerate}
\end{Corollary}
\begin{proof}(i) readily follows from Lemma \ref{diagonal}. To prove (ii), observe that, for a vector $v\in V$ of height $k$ one has $\omega(I^+_Z)v\otimes 1_\zeta=v\otimes \sigma(I^+_Z) 1_\zeta=v\otimes I^-_Z 1_\zeta =0$
modulo $(V\otimes Z)_{k-1}$. Clearly (iii) follows from~(ii). To prove~(iv), pick up $v\in V^+_Z$ of height $k$ and present $\theta_{V,Z}(v)$ as $\Upsilon_Z(v) \mod V^\perp_Z$. Then modulo $(V\otimes Z)_{k-1}$ we get the required:
\begin{gather*}
\theta_{V,Z}(v)\otimes 1_\zeta=\Upsilon_Z(v) \otimes 1_\zeta = \mathcal{F}(v\otimes 1_\zeta)=\delta_l(v).
\end{gather*}
Here the left equality follows from (ii), while the middle one is a consequence of Lemma~\ref{diagonal}.
\end{proof}

\section[Complete reducibility of $V\otimes Z$]{Complete reducibility of $\boldsymbol{V\otimes Z}$}
We derive a complete reducibility criterion from the following two lemmas capturing properties of $\theta_{V,Z}$ and therefore of the canonical form on $V\otimes Z$.

\begin{Lemma}\label{ext-tw-decomp}Suppose that $\theta_{V,Z}$ is surjective. Then $V\circledast Z=V\otimes Z$.
\end{Lemma}
\begin{proof} It is sufficient to prove the equality on weight subspaces. The statement is obvious for subspace of maximal weight $\nu+\zeta\in \Lambda(V\otimes Z)$. Pick up an arbitrary $\mu$ and suppose the statement is proved for all weights of height less than $k=\mathrm{ht}(\mu)$. Corollary \ref{two-filtrations} implies modulo $(V\otimes Z)_{k-1}[\mu]$:
\begin{gather*}
(V\otimes Z)[\mu]=(V\otimes Z)_k[\mu]={}^+\!V_Z[\mu-\zeta]\otimes 1_\zeta=\theta_{V,Z}V^+_Z[\mu-\zeta]\otimes 1_\zeta\\
\hphantom{(V\otimes Z)[\mu]}{} =(V\otimes Z)^+[\mu] =(V\circledast Z)[\mu].
\end{gather*}
Here we consecutively used Lemma \ref{two-filtrations}(i),~(iii), the surjectivity assumption, Lemma~\ref{two-filtrations}(iv), and, finally, (i) along with the induction assumption. This proves the statement for all weights of height $k$ and hence for all weights, by induction.
\end{proof}

\begin{Lemma}\label{irred} Suppose that $\theta_{V,Z} $ is injective. Then all submodules of highest weight in $V\otimes Z$ are irreducible.
\end{Lemma}
\begin{proof}Since $\theta_{V,Z} $ is injective, it is an isomorphism between $V^+_Z$ and ${}^+\!V_Z$. Then the canonical form is non-degenerate on $V\circledast Z$. If $X\subset Y$ are submodules of highest weight, then $X$ is orthogonal to all submodules in $V\circledast Z$ of the same weight, by Proposition~\ref{form_features}(iii),~(ii). Therefore~$X$ is orthogonal to all $V\circledast Z$, which is a contradiction.
\end{proof}

Now we can prove the main result of the paper.
\begin{Theorem} \label{dir-sum}The module $V\otimes Z$ is completely reducible if and only if the canonical form is non-degenerate on $(V\otimes Z)^+$.
\end{Theorem}
\begin{proof}We use the fact that the form is non-degenerate if and only if $\theta_{V,Z}$ is an isomorphism. If $V\otimes Z$ is completely reducible, then it is a sum of highest weight modules. The canonical form is non-degenerate on $V\circledast Z$ and therefore on $(V\otimes Z)^+$, by Proposition~\ref{form_features}(iii).

Conversely, if the form is non-degenerate, then all highest weight submodules are irreducible by Lemma~\ref{irred}. Since $\theta_{V,Z}$ is surjective, their sum $V\circledast Z$ exhausts all of $V\otimes Z$ by Lemma~\ref{ext-tw-decomp}.
\end{proof}

As an application of this criterion, we derive the following sufficient (and necessary) condition for complete reducibility.
\begin{Corollary}\label{dir-sum_irred} Suppose that all submodules in $V\otimes Z$ of highest weight are irreducible. Then $V\otimes Z$ is completely reducible.
\end{Corollary}
\begin{proof} It is sufficient to prove that $\theta_{V,Z}$ is injective and therefore surjective. That is obvious for its restriction to $\mathbb{C} 1_\nu\subset V^+_Z$, and further we proceed by induction on height.
Suppose that $v\in V^+_Z$ is a non-zero vector of height $k$ killed by $\theta_{V,Z}$. Then by Corollary \ref{two-filtrations}(iv), $\delta_l(v)\in (V\otimes Z)_{k-1}$ along with the submodule it generates. But that is impossible since, by induction assumption, $(V\otimes Z)_{k-1}$ is a direct sum of irreducible modules whose canonical generators have heights $< k$. Therefore $\theta_{V,Z}$ is an isomorphism on $ V^+_Z$ of height $k$, and $(V\otimes Z)_{k}$ is a sum of highest weight submodules, by Corollary~\ref{two-filtrations}(iii) and~(iv). This facilitates the induction transition and proves the assertion.
\end{proof}

\section{Special case of finite-dimensional modules}
Since finite-dimensional modules are completely reducible, so is $V\otimes Z$ once $\dim V<\infty$ and $\dim Z<\infty$. Although this case is not particularly interesting for this exposition, we include it for completeness and give an independent proof of complete reducibility via a sort of ``unitarian trick''.

Recall that a complex conjugation on a complex vector space $V$ is an $\mathbb{R}$-linear involution $V\to V$, $v\mapsto \bar v$ satisfying $\overline{c v}=\bar c \bar v$ for all $c\in \mathbb{C}$ and $v\in V$. It defines a~real structure on $V$, i.e., an $\mathbb{R}$-subspace $V_\mathbb{R}$ of stable vectors, such that $V=V_\mathbb{R}\otimes _\mathbb{R} \mathbb{C}$.

A bilinear form is called real if $\overline{\langle v,w \rangle }=\langle \bar v,\bar w \rangle $ for all $v,w\in V$. A real structure on $V$ establishes a bijection between bilinear and sesquilinear forms. For any bilinear form $\langle \cdot,\cdot\rangle$ the form $v\otimes w\mapsto \langle \bar v, w\rangle= (v, w)$ is anti-linear in the first argument and linear in the second. For a real symmetric form $\langle \cdot,\cdot\rangle$ the form $(\cdot,\cdot)$ is Hermitian. Clearly they both are degenerate and non-degenerate simultaneously.

Suppose that $q\in \mathbb{R}$ and define a complex conjugation on $U_q(\mathfrak{g})$ as an involutive antilinear Hopf algebra automorphism acting on the generators by
\begin{gather*}
\overline{q^{\pm h_\alpha}}=q^{\pm h_\alpha}, \qquad \bar f_\alpha =- f_\alpha, \qquad \bar e_\alpha= -e_\alpha.
\end{gather*}
We call a representation on a vector space $V$ with a real structure real if $\overline{x v}=\bar x \bar v$, for all $v\in V$ and all $x\in U_q(\mathfrak{g})$. Since $U_q(\mathfrak{b}_+)$ is invariant under the complex conjugation, it readily follows that a Verma module~$M_\lambda$ is real if and only if the highest weight $\lambda$ is in the real span of~$\Pi^+$. Then for integral dominant $\lambda$ all the submodules in $M_\lambda$ are real, as well as the irreducible finite-dimensional quotient of $M_\lambda$.

The compact real form on $U_q(\mathfrak{g})$ is an antilinear involutive algebra anti-automorphism and coalgebra automorphism~$*$. It is defined on the generators by
\begin{gather*}
(q^{\pm h_\alpha})^*=q^{\pm h_\alpha}, \qquad f_\alpha^*=q^{-h_\alpha} e_\alpha, \qquad e_\alpha^*= f_\alpha q^{h_\alpha}.
\end{gather*}
One can check that $\omega \circ *$ is the complex conjugation introduced above.

A representation on a module $V$ with a Hermitian form $( \cdot,\cdot )$ is called unitary if $(x v , w)=(v , x^* w )$ for all $v,w\in V$ and $x\in U_q(\mathfrak{g})$. In what follows, the Hermitian form $(\cdot ,\cdot )$ on a module with real structure is obtained from a $\omega$-contravariant form by the complex conjugation of the first argument. Then the representation automatically becomes unitary.

Consider an example of $(n+1)$-dimensional irreducible representation of $U_q (\mathfrak{sl}(2) )$ with the highest vector $1_n$. Then the Hermitian form is positive definite. Indeed,
\begin{gather*}
\big(f^m 1_n, f^m 1_n\big)=\bigl(1_n, \big(q^{-h}e\big)^m f^m1_n\bigr)\\
\hphantom{\big(f^m 1_n, f^m 1_n\big)}{} =\bigl(1_n, [m]_q! [h]_q[h-1]_q\cdots [h-m+1]_qq^{-mh+m(m+1)}1_n\bigr).
\end{gather*}
It is equal to $[n]_q[m-1]_q\cdots [n-m+1]_qq^{-mn+m(m+1)}[m]_q! (1_n,1_n)>0$ since $m\leqslant n$. In particular, the Hermitian operator $q^{-h}e f$ is non-negative and turns zero only on the lowest vector, that is, on $\ker f$.
\begin{Lemma} For any finite-dimensional $U_q(\mathfrak{g})$-module $V$, the Hermitian contravariant form on~$V$ relative to the compact real form on $U_q(\mathfrak{g})$ is positive definite.
\end{Lemma}
\begin{proof}Assume for simplicity that $V$ is irreducible. Let $V_k\subset V$ denote the vector subspace spanned by vectors of height $\leqslant k$. We need to prove that the form is positive definite on all~$V_k$. It is so on $V_0$ by construction as it is normalized to $(1_\nu,1_\nu)=1$ on the highest vector. Assume the lemma proved for $V_k$ with $k\geqslant 0$ and pick up some non-zero $v$ of height $k+1$. Then $v=f_\alpha w\not =0$ for some $\alpha \in \Pi^+$ and $w\in V_k$. By induction assumption, the form is positive definite on $V_k$, so we can assume that $w$ is orthogonal to $\ker f_\alpha\cap V_{k}$. As $q^{-h_\alpha}e_\alpha f_\alpha$ is positive on the span of such vectors, $(f_\alpha w,f_\alpha w)=(w, q^{-h_\alpha}e_\alpha f_\alpha w)>0$. Therefore the form is positive definite on $V_{k+1}$ and on all $V$, by the induction argument.
\end{proof}

If the modules $V$ and $Z$ are equipped with a real structure, then so is $V\otimes Z$, and $\overline{(V\otimes Z)^+}=(V\otimes Z)^+$. This implies that the restriction of the Hermitian form to $(V\otimes Z)^+$ is obtained from the restriction of the canonical form by the conjugation of the first argument.
\begin{Proposition}Suppose that highest weight modules $V$ and $Z$ are finite-dimensional. Then the canonical form on $V\otimes Z$ is non-degenerate on $(V\otimes Z)^+$.
\end{Proposition}
\begin{proof} The corresponding Hermitian form is positive definite on $V\otimes Z$ and therefore non-degenerate on $(V\otimes Z)^+$. Therefore the canonical bilinear form is non-degenerate on $(V\otimes Z)^+$ as well.
\end{proof}

\section{On computing extremal twist}
Observe that the modules $V$ and $Z$ enter the tensor product in a symmetric way. Therefore one can consider the operator $\theta_{Z,V}\colon \to {}^+\!Z_V$, which comes from the alternative parametrization of singular vectors with elements of $ Z^+_V$. The operators $\theta_{V,Z}$ and $\theta_{Z,V}$ are equivalent in the sense that they are simultaneously invertible. However, from the technical point of view, one of them may be easier to work with than the other. Such an example can be found in~\cite{M1}, where one of the spaces $ V^+_Z$, $ Z^+_V$ is easy to identify while the other is unknown.

The canonical form on $V\otimes Z$ translates the question of complete reducibility to non-dege\-ne\-ra\-cy of a matrix, which is finite-dimensional when restricted to each weight subspace. For computational purposes, one can work either with $\theta_{V,Z}$ or with $\theta_{Z,V}$, depending on a particular situation. For a special case when $Z$ is a Verma module, $\theta^{-1}_{V,Z}$ is expressed in~\cite{EV} through the Zhelobenko cocycle~\cite{Zh} via dynamical twist, which is closely related to $\mathcal{F}$. In this section we will show how the result of \cite{EV} can be used in a more general situation.

 Observe that calculation of $\theta_{V,Z}$ through $\Upsilon_Z$ involves more information than required for a particular $V$. In fact, given a singular vector $u=\sum_i v_i\otimes z_i$, we need to lift only its $Z$-coefficients: $z_i\mapsto f_i\in U_q(\mathfrak{g}_-)$. Then $\theta_{V,Z}(v)= \sum_{i}\gamma^{-1}(f_i)v_i$, where $v\in V$ is the leading coefficient. This observation leads to the following method of finding $\theta_{V,Z}$.

\looseness=-1 For a Verma module $\hat M_\zeta$, the operator $\theta_{V,\hat M_\zeta}$ is a rational trigonometric function of $\zeta$ and may have poles when $\hat M_\zeta$ becomes reducible. Still we can consider it on a subspace in~$V$, where it is regular. Then we can use $\theta_{V,\hat M_\zeta}$ to calculate $\theta_{V,Z}$ through a limit procedure, as explained below.

 Let $V^+_{\zeta}$ denote the vector space spanned by leading $V$-coefficients of singular vectors in $\big(V\otimes \hat M_\zeta\big)^+$. Chose a section $\delta'_l$ of the map $\big(V\otimes \hat M_\zeta\big)^+\to V^+_\zeta$ assigning the leading coefficient to a singular vector. For $v\in V^+_{\zeta}$ pick up a presentation $u=\sum_i v_i\otimes f_i1_\zeta =\delta'_l(v)\in V\otimes \hat M_\zeta$. We define $\theta_{V,\hat M_\zeta}\colon V^+_\zeta\to {}^+\!V_Z$ by projecting $\sum_{i}\gamma^{-1}(f_i)v_i$ to ${}^+\!V_Z$ along $V^\perp_Z$ and assigning the result to~$v$. Note that it is independent of the choice of $\delta'_l$. Indeed, $\delta'_l$ delivers singular vectors in $V\otimes \hat M_\zeta$ that survive in $V\otimes Z$. If a singular vector $u=\sum_i v_i\otimes f_i1_\zeta$ vanishes in $V\otimes Z$, then all $f_i1_\zeta$ belong to $I^-_Z1_\zeta$. But then $\sum_{i}\gamma^{-1}(f_i)v$ is killed in projection to ${}^+\!V_Z$.

\begin{Proposition}\label{ex_twist_lift}Suppose that all singular vectors in $V\otimes Z$ are images of singular vectors in $V\otimes \hat M_\zeta$. Then $V^+_Z= V^+_{\zeta}$ and $\theta_{V,Z}=\theta_{V,\hat M_\zeta}$.
 \end{Proposition}
\begin{proof} The image of $\delta_l'$ consists of singular vectors that survive in projection $V\otimes \hat M_\zeta\to V\otimes Z$, hence the equality $V^+_Z= V^+_{\zeta}$.

Let $u\in V\otimes Z$ be a singular vector with leading coefficient $v\in V^+_Z$. By the hypothesis, it is the image of a singular vector $\hat u\in V\otimes \hat M_\zeta$ with the same leading coefficient. But then $\theta_{V,\hat M_\zeta}(v)=\theta_{V,Z}(v)$, because a lift of $\hat u$ is a lift of $u$.
\end{proof}

\subsection*{Acknowledgements}

We are grateful to Joseph Bernstein for useful discussions. We are also indebted to the anonymous referees for careful reading of the text and valuable remarks. This study was supported by the RFBR grant 15-01-03148.

\pdfbookmark[1]{References}{ref}
\LastPageEnding

\end{document}